\documentclass[10pt,reqno]{amsart}
\usepackage{amsfonts}
\usepackage{amssymb}
\usepackage{amscd}
\usepackage{hyperref}

\newcommand{\func}[1]{\operatorname{#1}}

\newtheorem{theorem}{Theorem}[section]
\newtheorem{corollary}[theorem]{Corollary}
\newtheorem{lemma}[theorem]{Lemma}
\newtheorem{proposition}[theorem]{Proposition}

\newtheorem{remark}[theorem]{Remark}

\newtheorem{example}[theorem]{Example}

\numberwithin{equation}{section}
\begin{document}
\title{Riemannian Flows and Adiabatic Limits}
\author[G.~Habib]{Georges Habib}
\address{Lebanese University \\
Faculty of Sciences II \\
Department of Mathematics\\
P.O. Box 90656 Fanar-Matn \\
Lebanon}
\email[G.~Habib]{ghabib@ul.edu.lb}
\author[K.~Richardson]{Ken Richardson}
\address{Department of Mathematics \\
Texas Christian University \\
Fort Worth, Texas 76129, USA}
\email[K.~Richardson]{k.richardson@tcu.edu}
\thanks{This work was supported by a grant from the Simons Foundation (Grant
Number 245818 to Ken Richardson), the Alexander von Humboldt Foundation, Institut f\"ur Mathematik der
Universit\"at Potsdam, and Centro Internazionale per la Ricerca Matematica
(CIRM)}

\begin{abstract}
We show the convergence properties of the eigenvalues of the Dirac operator
on a spin manifold with a Riemannian flow when the metric is collapsed along
the flow.
\end{abstract}

\maketitle
\tableofcontents

\section{Introduction}

Many researchers have studied the spectrum of the Laplacian and 
Dirac-type operators on families of manifolds where the metric is collapsed. 
We point out in particular the references
\cite{Col}, \cite{Fuk}, \cite{GLP}, where the behavior of the spectrum of
Laplacians on Riemannian submersions are noted under collapse of the
fiber metrics.
In \cite{MaMe}, R. R. Mazzeo and R. B. Melrose related the properties of
the Laplace eigenvalues under adiabatic limits in a Riemannian fiber 
bundle to the Leray spectral sequence, 
and J. A. \'{A}lvarez-L\'{o}pez and Y. Kordyukov extended this analysis in \cite{ALK}
to the more general case of Riemannian foliations; see 
\cite{KY} for an exposition and further references.
Adiabatic limits of the eta invariants of Dirac operators have also been considered, 
as in \cite{Wit}, \cite{BiC}, and \cite{D1}.

In \cite{Am-Ba}, B. Ammann and C. B\"{a}r examined the eigenvalues of the
Dirac operator of circle bundles over a closed Riemannian manifold $M\diagup
S^{1}$, such that the bundle projection is a Riemannian submersion. They
found that as the metric is changed such that the lengths of the circles
collapse to zero, the eigenvalues separate into two categories: those that
converge to the eigenvalues of the base (quotient) manifold which correspond
to the \textit{projectable spinors} --- for which the Lie derivative is zero
in the direction of the fibers --- and those eigenvalues that go to
infinity, corresponding to \textit{non-projectable spinors}. The main idea
is to decompose the Lie derivative of any spinor field on $M$ into
finite-dimensional eigenspaces $V_{k}$ ($k\in \mathbb{Z}$), and such a
decomposition is preserved by the Dirac operator. This comes from the
representation of the Lie group $S^{1}$ on the spinor bundle on $M$. In a
second step, they decompose the Dirac operator of the whole manifold $M$
into a horizontal and vertical Dirac operator and a zero$^{\text{th}}$ order
term. It turns out that the horizontal Dirac operator commutes with the Lie
derivative, while the vertical part anticommutes. This allows the
researchers to compute explicitly the eigenvalues of the Dirac operator on $M
$ on each eigenspace $V_{k}$ in terms of $k$. Here the zero$^{\text{th}}$
order term does not contribute in the adiabatic limit, since it is a bounded
operator and tends to zero with the length of the fibers. 
In \cite{Am}, B. Ammann extended the result above to the case where
the circles form a more general Riemannian submersion with projectable
spin structures over a base manifold.
Also, in \cite{Pfa}, F. Pf\"affle studied the degeneration of Dirac eigenvalues in a 
sequence of compact spin hyperbolic manifolds in the case the limit has
discrete Dirac spectrum.
We also mention
the work of J. Lott in \cite{Lott}, where the limit of a general
Dirac-type operator is studied under a collapse for which the diameter and
sectional curvature are bounded. In this case, the spectrum of the Dirac
operator converges to the spectrum of a limiting first order operator. 

In this paper, we consider a particular case of foliations, namely \textit{%
Riemannian flows}. On a Riemannian manifold $(M,g),$ a Riemannian flow is a
foliation of $1$-dimensional leaves given by the integral curves of a unit
vector field $\xi $ such that $g$ is a bundle-like metric. This means the
Lie derivative of the transverse metric in the direction of $\xi $ vanishes.
Examples of such flows are those given by Killing vector fields and Sasakian
manifolds. Those are called taut (meaning the mean curvature form is exact),
but examples of nontaut Riemannian flows exist (see, for example, \cite{Car}%
).\newline

We now take the adiabatic limit of the Riemannian flow, and in our 
situation it is often not the case that the limit is a manifold. This means we
consider the bundle-like metric 
\begin{equation*}
g_{f}=f^{2}\xi ^{\ast }\otimes \xi ^{\ast }+g_{\xi ^{\bot }},
\end{equation*}%
where $f$ is a positive basic function on $M$, and we prove that the
eigenvalues of the Dirac operator on $(M,g_{f})$ corresponding to basic
sections tend to those of the basic Dirac operator $D_{b}$, which is morally
the Dirac operator of the local quotients in the foliation charts; see the
next section for details. We point out that our case does not require the
leaves to be circles, unlike the situation in \cite{Am-Ba} or in \cite{Am}. Also, we prove
that when the flow is taut, the eigenvalues from the $L^{2}$-orthogonal
complement of the space of basic sections of the spinor bundle go to $\pm
\infty $. The main difference between our case and the one in \cite{Am-Ba}
is that there is not necessarily a circle action on the manifold $M$, which
mainly means that the $L^{2}$-decomposition of the Lie derivative in the
direction of the flow cannot carry over. Moreover, the leafwise Dirac
operator could fail to have discrete spectrum. We also mention the work of
P. Jammes in \cite{Jam}, where he considered adiabatic limits of Riemannian flows,
similar to our setting, and examined their effect on the eigenvalues of the
Laplacian.

In Section \ref{PrelimSection}, we provide preliminary details on spin
Riemannian flows and in particular define the leafwise Dirac operator $D_{%
\mathcal{F}}$ and the symmetric transversal Dirac operator $D_{Q}$ ($Q=\xi
^{\bot }=N\mathcal{F}$).\ In Lemma \ref{CommutatorLemma}, we express the
anticommutator of these operators in terms of the mean curvature. We show
the operator $D_{\mathcal{F}}$ is symmetric, and its kernel is the $L^{2}$%
-closure of the space of basic sections (see Proposition \ref{DFSymmProp}).
In Corollary \ref{minimalCaseDFeigenvSpectrumCor}, we prove that when the
flow is minimal, the spectrum of $D_{\mathcal{F}}$ contains a countable
number of real eigenvalues, and there exists a complete orthonormal basis of
the $L^{2}$ spinors consisting of smooth eigensections of $D_{\mathcal{F}}$ .

Our main result is Theorem \ref{CollapsingThm}, where we show that the
eigenvalues behave as stated above in the adiabatic limit. In Section \ref%
{examplesSection}, we exhibit examples which show interesting behavior of
the operators $D_{\mathcal{F}}$ and $D_{Q}$. In these examples, which are
not fibrations, the operator $D_{\mathcal{F}}$ does not have discrete
spectrum, but nonetheless the conclusion of the main theorem is made clear.

\section{Dirac operators on Riemannian flows\label{PrelimSection}}

Let $\left( M,g\right) $ be a closed $\left( n+1\right) $-dimensional
Riemannian manifold, endowed with an oriented Riemannian flow. This means
that there exists a unit vector field $\xi $ on $M$ such that the Lie
derivative of the transverse metric vanishes: $\mathcal{L}_{\xi }\left(
\left. g\right\vert _{\xi ^{\bot }}\right) =0$ (see \cite{Rein}, \cite{Car}, 
\cite{Tond}). Suppose in addition that $M$ is spin, and let $D_{M}$ be the
Dirac operator associated to the spin structure acting on sections of the
spinor bundle $\Sigma M$, which has a given hermitian metric and metric spin
connection.

We wish to construct the basic Dirac operator associated to the induced spin
structure on the normal bundle. Since $TM=\mathbb{R}\xi \oplus \xi ^{\bot }$%
, the pullback of the spin structure on $M$ induces a spin structure on the
normal bundle $Q=\xi ^{\bot }=N\mathcal{F}$. In this case, the spinor bundle 
$\Sigma M$ is canonically identified with the spinor bundle $\Sigma Q$ of $Q$%
, for $n$ even, and with the direct sum $\Sigma Q\oplus \Sigma Q$ for $n$
odd. The metric on $\Sigma M$ induces a metric on $\Sigma Q$. When $n$ is
even, then $i\xi \cdot $ is taken to be the chirality operator, as $\left(
i\xi \cdot \right) ^{2}=\mathrm{id}$, and we let $\left( \Sigma Q\right)
^{\pm }$ be the eigenspaces associated to the $\pm 1$ eigenvalues, with
Clifford multiplication $\cdot _{Q}$ defined by $Z\cdot _{Q}\varphi =Z\cdot
_{M}\varphi $ for $Z\in \Gamma \left( Q\right) $, $\varphi \in \Gamma \left(
\Sigma Q\right) $. When $n$ is odd, the Clifford multiplications $\cdot _{M}$
on $\Sigma M$ and $\cdot _{Q}$ on $\Sigma Q:=\Sigma M^{+}$ are related by $%
Z\cdot _{Q}\varphi =Z\cdot _{M}\xi \cdot _{M}\varphi $ (as in \cite{Baer}).
Therefore, by using the above identification, the spinor connections $\nabla
^{\Sigma M}$ and $\nabla ^{\Sigma Q}$ are related by the following relations
(see \cite[formula 4.8]{HabThesis}). For all $Z\in \Gamma \left( Q\right) $, 
\begin{eqnarray}
\nabla _{\xi }^{\Sigma M}\varphi  &=&\nabla _{\xi }^{\Sigma Q}\varphi +\frac{%
1}{2}\Omega \cdot _{M}\varphi +\frac{1}{2}\xi \cdot _{M}\kappa \cdot
_{M}\varphi ,  \notag \\
\nabla _{Z}^{\Sigma M}\varphi  &=&\nabla _{Z}^{\Sigma Q}\varphi +\frac{1}{2}%
\xi \cdot _{M}\left( \nabla _{Z}^{M}\xi \right) \cdot _{M}\varphi ,
\label{covDerviEqs}
\end{eqnarray}%
where the Euler form $\Omega $ is the $2$-form given for all $Y,Z\in \Gamma
\left( Q\right) $ by $\Omega \left( Y,Z\right) =g\left( \nabla _{Y}^{M}\xi
,Z\right) $ and $\kappa ^{\#}=\nabla _{\xi }^{M}\xi \in \Gamma \left(
Q\right) $ is the mean curvature vector field of the flow. The one-form $%
\kappa $ is also identified with the corresponding Clifford algebra element.
We identify $\Omega $ with the associated element of the Clifford algebra by 
$\Omega =\frac{1}{2}\sum_{j=1}^{n}e^{j}\wedge \left( \nabla _{e_{j}}^{M}\xi
\right) ^{\flat }=\frac{1}{2}\sum_{j=1}^{n}e_{j}\cdot _{M}\left( \nabla
_{e_{j}}^{M}\xi \right) \cdot _{M}$ where here and in the following $\left\{
e_{j}\right\} _{j=1}^{n}$ is a local orthonormal frame of $\Gamma \left(
Q\right) $.

\begin{lemma}
\label{curvatureVanishingLemma}(in \cite{HabThesis}) If $K\left( X,Y\right)
=X\cdot Y\cdot \left( \nabla _{X}^{\Sigma Q}\nabla _{Y}^{\Sigma Q}-\nabla
_{Y}^{\Sigma Q}\nabla _{X}^{\Sigma Q}+\nabla _{\left[ X_{i},Y\right]
}^{\Sigma Q}\right) $ is the Clifford curvature of $\Sigma Q$, then $K\left(
X,Y\right) =0$ if $X=\xi $.
\end{lemma}

\begin{lemma}
\label{xiCommutesWithCovDeriv}The transverse connection commutes with the
Clifford action of $\xi $; that is, $\nabla _{X}^{\Sigma Q}\left( \xi \cdot
_{M}\varphi \right) =\xi \cdot _{M}\nabla _{X}^{\Sigma Q}\varphi $ for any
spinor field $\varphi \in \Gamma \left( \Sigma Q\right) $ and $X\in \Gamma
\left( TM\right) $. In particular, this means that the spinor field $\xi
\cdot _{M}\varphi $ is basic if and only if $\varphi $ is basic.

\begin{proof}
We use (\ref{covDerviEqs}). For $Z\in \Gamma \left( Q\right) $,%
\begin{eqnarray*}
\nabla _{Z}^{\Sigma Q}\left( \xi \cdot _{M}\varphi \right) &=&\nabla
_{Z}^{\Sigma M}\left( \xi \cdot _{M}\varphi \right) -\frac{1}{2}\xi \cdot
_{M}\nabla _{Z}^{M}\xi \cdot _{M}\xi \cdot _{M}\varphi \\
&=&\left( \nabla _{Z}^{\Sigma M}\xi \right) \cdot _{M}\varphi +\xi \cdot
_{M}\nabla _{Z}^{\Sigma M}\varphi -\frac{1}{2}\left( \nabla _{Z}^{M}\xi
\right) \cdot _{M}\varphi \\
&=&\xi \cdot _{M}\nabla _{Z}^{\Sigma M}\varphi -\frac{1}{2}\xi \cdot _{M}\xi
\cdot _{M}\left( \nabla _{Z}^{M}\xi \right) \cdot _{M}\varphi \\
&=&\xi \cdot _{M}\nabla _{Z}^{\Sigma Q}\varphi ,
\end{eqnarray*}%
since $\nabla _{Z}^{M}\xi $ is orthogonal to $\xi $. Next, 
\begin{eqnarray*}
\nabla _{\xi }^{\Sigma Q}\left( \xi \cdot _{M}\varphi \right) &=&\nabla
_{\xi }^{\Sigma M}\left( \xi \cdot _{M}\varphi \right) -\frac{1}{2}\Omega
\cdot _{M}\xi \cdot _{M}\varphi -\frac{1}{2}\xi \cdot _{M}\kappa \cdot
_{M}\xi \cdot _{M}\varphi \\
&=&H\cdot _{M}\varphi +\xi \cdot _{M}\nabla _{\xi }^{\Sigma M}\varphi -\frac{%
1}{2}\xi \cdot _{M}\Omega \cdot _{M}\varphi -\frac{1}{2}\kappa \cdot
_{M}\varphi \\
&=&\xi \cdot _{M}\nabla _{\xi }^{\Sigma M}\varphi -\frac{1}{2}\xi \cdot
_{M}\Omega \cdot _{M}\varphi -\frac{1}{2}\xi \cdot _{M}\xi \cdot _{M}\kappa
\cdot _{M}\varphi \\
&=&\xi \cdot _{M}\nabla _{\xi }^{\Sigma Q}\varphi .
\end{eqnarray*}
\end{proof}
\end{lemma}

\vspace{0in}We define the transversal Dirac operator $D_{Q}$ on $\Gamma
\left( \Sigma Q\right) $ as%
\begin{equation*}
D_{Q}=\sum_{i=1}^{n}e_{i}\cdot _{Q}\nabla _{e_{i}}^{\Sigma Q}-\frac{1}{2}%
\kappa \cdot _{Q}.
\end{equation*}%
This differential operator is first-order and transversally elliptic. Using
the metric on $\Sigma Q$ induced from the metric on $\Sigma M$, we obtain
the $L^{2}$ metric on $\Gamma \left( \Sigma Q\right) $.

\begin{lemma}
(From \cite[p. 31]{HabThesis}) The operator $D_{Q}$ is self-adjoint on $%
L^{2}\left( \Gamma \left( \Sigma Q\right) \right) $.
\end{lemma}

The basic Dirac operator $D_{b}$ is the restriction of 
\begin{equation*}
PD_{Q}=\sum_{i=1}^{n}e_{i}\cdot _{Q}\nabla _{e_{i}}^{\Sigma Q}-\frac{1}{2}%
\kappa _{b}\cdot _{Q}
\end{equation*}%
to the set $\Gamma _{b}\left( \Sigma Q\right) $ of basic sections (sections $%
\varphi $ in $\Gamma \left( \Sigma Q\right) $ satisfying $\nabla _{\xi
}^{\Sigma Q}\varphi =0$):%
\begin{equation*}
D_{b}=\left. PD_{Q}\right\vert _{\Gamma _{b}\left( \Sigma Q\right) }.
\end{equation*}%
In the above, $P:L^{2}\left( \Gamma \left( \Sigma Q\right) \right)
\rightarrow L^{2}\left( \Gamma _{b}\left( \Sigma Q\right) \right) $ is the
orthogonal projection onto basic sections, and $\kappa _{b}=P_{b}\kappa $
where $P_{b}:L^{2}\left( \Omega ^{\ast }\left( M\right) \right) \rightarrow
L^{2}\left( \Omega _{b}^{\ast }\left( M\right) \right) $ (see \cite{AL}, 
\cite{PaRi}, \cite{BKR2}). It is always true that $P$ preserves the smooth
sections and that $\kappa _{b}$ is a closed one-form. Recall that the basic
Dirac operator preserves the set of basic sections and is transversally
elliptic and essentially self-adjoint (on the basic sections). Therefore, by
the spectral theory of transversally elliptic operators, it is a Fredholm
operator and has discrete spectrum (\cite{EK}, \cite{EKG}). Observe that
when $\kappa $ is a basic form, 
\begin{equation*}
\kappa _{b}=\kappa \text{,~}D_{b}=\left. D_{Q}\right\vert _{\Gamma
_{b}\left( \Sigma Q\right) }.
\end{equation*}%
If the mean curvature is not necessarily basic, then 
\begin{eqnarray*}
D_{Q} &=&\sum_{i=1}^{n}e_{i}\cdot _{Q}\nabla _{e_{i}}^{\Sigma Q}-\frac{1}{2}%
\kappa \cdot _{Q} \\
&=&\sum_{i=1}^{n}e_{i}\cdot _{Q}\nabla _{e_{i}}^{\Sigma Q}-\frac{1}{2}\kappa
_{b}\cdot _{Q}+\frac{1}{2}\left( \kappa _{b}-\kappa \right) \cdot _{Q} \, ,\\
\left. D_{Q}\right\vert _{\Gamma _{b}\left( \Sigma Q\right) } &=&D_{b}+\frac{%
1}{2}\left( \kappa _{b}-\kappa \right) \cdot _{Q}\, .
\end{eqnarray*}

Next, we give the relationship between $D_{M}$ and $D_{Q}$ on $\Gamma \left(
\Sigma M\right) $. By (\ref{covDerviEqs}) we have%
\begin{eqnarray}
D_{M} &=&D_{Q}-\frac{1}{2}\xi \cdot _{M}\Omega \cdot _{M}+\xi \cdot
_{M}\nabla _{\xi }^{\Sigma Q}\text{ for }n\text{ even,}  \notag \\
D_{M} &=&\xi \cdot _{M}\left( D_{Q}\oplus \left( -D_{Q}\right) \right) -%
\frac{1}{2}\xi \cdot _{M}\Omega \cdot _{M}+\xi \cdot _{M}\left( \nabla _{\xi
}^{\Sigma Q\oplus \Sigma Q}\right) \text{ for }n\text{ odd.}
\label{D_mD_tr formula}
\end{eqnarray}%
Using the formulas above, the restrictions of the Dirac operators $D_{M}$
and $D_{b}$ to basic sections are related by%
\begin{eqnarray}
\left. D_{M}\right\vert _{\Gamma _{b}\left( \Sigma Q\right) } &=&D_{b}+\frac{%
1}{2}\left( \kappa _{b}-\kappa \right) \cdot _{Q}-\frac{1}{2}\xi \cdot
_{M}\Omega \cdot _{M}\text{ for }n\text{ even,}  \notag \\
\left. D_{M}\right\vert _{\Gamma _{b}\left( \Sigma Q\right) } &=&\xi \cdot
_{M}D_{b}+\frac{1}{2}\xi \cdot _{M}\left( \kappa _{b}-\kappa \right) \cdot
_{Q}-\frac{1}{2}\xi \cdot _{M}\Omega \cdot _{M}\text{ for }n\text{ odd.}
\label{D_MD_b formula}
\end{eqnarray}

For $n$ even, respectively $n$ odd, and for any basic spinor field $\varphi $%
, we have that $D_{b}\left( \xi \cdot _{M}\varphi \right) =-\xi \cdot
_{M}D_{b}\left( \varphi \right) $, respectively $D_{b}\left( \xi \cdot
_{M}\varphi \right) =\xi \cdot _{M}D_{b}\left( \varphi \right) $. Hence, the
spectrum of $D_{b}$ is symmetric about $0$ for $n$ even.

Observe that Rummler's formula is%
\begin{eqnarray*}
d\left( \xi ^{\ast }\right) &=&-\kappa \wedge \xi ^{\ast }+\varphi _{0} \\
&=&\xi ^{\ast }\wedge \nabla _{\xi }^{M}\xi ^{\ast
}+\sum_{j=1}^{n}e^{j}\wedge \nabla _{e_{j}}^{M}\xi ^{\ast } \\
&=&-\kappa \wedge \xi ^{\ast }+2\Omega ,
\end{eqnarray*}%
so that $\varphi _{0}=2\Omega $. Since $\varphi _{0}$ is always of type $%
\left( 2,0\right) $ in $\Lambda ^{\ast }Q\wedge \Lambda ^{\ast }T\mathcal{F}$
for flows, we see $\Omega \in \Gamma \left( M,\Lambda ^{2}Q\right) $.

\begin{lemma}
If $\kappa $ is a basic form, then $\Omega $ is basic.

\begin{proof}
We see that%
\begin{equation*}
i_{\xi }\Omega =\frac{1}{2}i_{\xi }\left( d\left( \xi ^{\ast }\right)
+\kappa \wedge \xi ^{\ast }\right) =0,
\end{equation*}%
which is clear since $\varphi _{0}=2\Omega $ is of type $\left( 2,0\right) $
in $\Lambda ^{\ast }Q\wedge \Lambda ^{\ast }T\mathcal{F}$. Next, since $%
\kappa $ is a basic closed form,%
\begin{eqnarray*}
i_{\xi }d\Omega &=&\frac{1}{2}i_{\xi }\left( \left( d\kappa \right) \wedge
\xi ^{\ast }-\kappa \wedge d\left( \xi ^{\ast }\right) \right) \\
&=&\frac{1}{2}i_{\xi }\left( -\kappa \wedge \left( -\kappa \wedge \xi ^{\ast
}+2\Omega \right) \right) \\
&=&\frac{1}{2}i_{\xi }\left( -\kappa \wedge \left( 2\Omega \right) \right)
=0.
\end{eqnarray*}
\end{proof}
\end{lemma}

\begin{remark}
The calculation above also shows that in the case where $\kappa $ is not
necessarily basic, 
\begin{equation*}
i_{\xi }d\Omega =\frac{1}{2}d_{1,0}\kappa =\frac{1}{2}d_{1,0}\left( \kappa
-\kappa _{b}\right) .
\end{equation*}
\end{remark}

\vspace{0in}For the case when $\kappa =\kappa _{b}$, by the equations above
for $D_{M}$ when $n$ is even, we see that $D_{M}$ preserves the basic
sections of $\Sigma M=\Sigma Q$, and since $D_{M}$ is orthogonally
diagonalizable over $L^{2}\left( \Sigma M\right) =L^{2}\left( \Sigma
Q\right) $, there exists an orthonormal basis of $L^{2}\left( \Gamma
_{b}\left( \Sigma Q\right) \right) $ consisting of eigensections of $D_{M}$.
Similarly, there exists an orthonormal basis of $L^{2}\left( \Gamma
_{b}\left( \Sigma Q\right) \right) ^{\bot }$ consisting of eigensections of $%
D_{M}$. The analogous facts are true for $n$ odd and $\left.
D_{M}\right\vert _{\Gamma _{b}\left( \Sigma Q\oplus \Sigma Q\right) }$ and $%
\left. D_{M}\right\vert _{\left( \Gamma _{b}\left( \Sigma Q\oplus \Sigma
Q\right) \right) ^{\bot }}$.We have shown the following.

\begin{lemma}
Suppose that $\kappa $ is basic. Then the operator $D_{M}$ decomposes as $%
\left. D_{M}\right\vert _{\Gamma _{b}\left( \Sigma Q\right) }\oplus \left.
D_{M}\right\vert _{\Gamma _{b}\left( \Sigma Q\right) ^{\bot }}$ as an $L^{2}$%
-orthogonal direct sum, when $n$ is even. It decomposes as $\left.
D_{M}\right\vert _{\Gamma _{b}\left( \Sigma Q\oplus \Sigma Q\right) }\oplus
\left. D_{M}\right\vert _{\left( \Gamma _{b}\left( \Sigma Q\oplus \Sigma
Q\right) \right) ^{\bot }}$ when $n$ is odd.
\end{lemma}

We call the operator $D_{\mathcal{F}}:=\xi \cdot _{M}\nabla _{\xi }^{\Sigma
Q}$ acting on $\Gamma \left( \Sigma Q\right) $ the \textbf{tangential Dirac
operator.}

\begin{proposition}
\label{DFSymmProp}The operator $D_{\mathcal{F}}$ is symmetric, and $\ker D_{%
\mathcal{F}}=L^{2}\left( \Gamma _{b}\left( \Sigma Q\right) \right) $.

\begin{proof}
For any (smooth) spinor fields $\psi $ and $\varphi $, letting $\left(
\bullet ,\bullet \right) $ be the pointwise inner product,%
\begin{eqnarray*}
\left( D_{\mathcal{F}}\psi ,\varphi \right) &=&\left( \xi \cdot _{M}\nabla
_{\xi }^{\Sigma Q}\psi ,\varphi \right) \\
&=&\left( \nabla _{\xi }^{\Sigma Q}\left( \xi \cdot _{M}\psi \right)
,\varphi \right)
\end{eqnarray*}%
by Lemma \ref{xiCommutesWithCovDeriv}. Then%
\begin{eqnarray*}
\left( D_{\mathcal{F}}\psi ,\varphi \right) &=&\xi \left( \xi \cdot _{M}\psi
,\varphi \right) -\left( \xi \cdot _{M}\psi ,\nabla _{\xi }^{\Sigma
Q}\varphi \right) \\
&=&\xi \left( \xi \cdot _{M}\psi ,\varphi \right) +\left( \psi ,\xi \cdot
_{M}\nabla _{\xi }^{\Sigma Q}\varphi \right) \\
&=&\xi \left( \xi \cdot _{M}\psi ,\varphi \right) +\left( \psi ,D_{\mathcal{F%
}}\varphi \right) .
\end{eqnarray*}%
Observe that, letting $f$ be the function $f=\left( \xi \cdot _{M}\psi
,\varphi \right) $, 
\begin{equation*}
\int_{M}\xi \left( f\right) =-\int_{M}f\func{div}\left( \xi \right) =0,
\end{equation*}%
since $\xi $ generates a Riemannian flow and thus is divergence-free. Thus,
by integrating $\left\langle D_{\mathcal{F}}\psi ,\varphi \right\rangle
=\left\langle \psi ,D_{\mathcal{F}}\varphi \right\rangle $. Next, if $D_{%
\mathcal{F}}\left( \varphi \right) =0$ for some section $\varphi \in \Gamma
\left( \Sigma Q\right) $, then 
\begin{equation*}
\xi \cdot _{M}0=\xi \cdot _{M}\xi \cdot _{M}\nabla _{\xi }^{\Sigma Q}\varphi
=-\nabla _{\xi }^{\Sigma Q}\varphi ,
\end{equation*}%
so $\varphi $ is basic.
\end{proof}
\end{proposition}

\begin{lemma}
\label{CommutatorLemma}We have $D_{Q}D_{\mathcal{F}}=-D_{\mathcal{F}%
}D_{Q}+\kappa \cdot _{M}D_{\mathcal{F}}=-D_{\mathcal{F}}\left( D_{Q}+\kappa
\cdot _{M}\right) $.
\end{lemma}

\begin{proof}
We see that, letting $e_{1},...,e_{n}$ be a local orthonormal frame for $Q$, 
\begin{eqnarray*}
D_{Q}\left( \xi \cdot _{M}\nabla _{\xi }^{\Sigma Q}\right)  &=&\left(
\sum_{i=1}^{n}e_{i}\cdot _{Q}\nabla _{e_{i}}^{\Sigma Q}-\frac{1}{2}\kappa
\cdot _{Q}\right) \left( \xi \cdot _{M}\nabla _{\xi }^{\Sigma Q}\right)  \\
&=&\sum_{i=1}^{n}\left( e_{i}\cdot _{M}\xi \cdot _{M}\nabla _{e_{i}}^{\Sigma
Q}\nabla _{\xi }^{\Sigma Q}\right) -\frac{1}{2}\kappa \cdot _{M}\xi \cdot
_{M}\nabla _{\xi }^{\Sigma Q},
\end{eqnarray*}%
by Lemma \ref{xiCommutesWithCovDeriv}. Then%
\begin{equation*}
D_{Q}\left( \xi \cdot _{M}\nabla _{\xi }^{\Sigma Q}\right)
=\sum_{i=1}^{n}\left( K\left( e_{i},\xi \right) +e_{i}\cdot _{M}\xi \cdot
_{M}\left( \nabla _{\xi }^{\Sigma Q}\nabla _{e_{i}}^{\Sigma Q}+\nabla _{
\left[ e_{i},\xi \right] }^{\Sigma Q}\right) \right) -\frac{1}{2}\kappa
\cdot _{M}\xi \cdot _{M}\nabla _{\xi }^{\Sigma Q}.
\end{equation*}

By Lemma \ref{curvatureVanishingLemma}, $K\left( e_{i},\xi \right) =0$ for
every $i$. Note that $\left[ e_{i},\xi \right] \in T\mathcal{F}$ so that%
\begin{eqnarray*}
\left[ e_{i},\xi \right] &=&\left\langle \left[ e_{i},\xi \right] ,\xi
\right\rangle \xi =\left\langle \nabla _{e_{i}}\xi -\nabla _{\xi }e_{i},\xi
\right\rangle \xi \\
&=&\frac{1}{2}e_{i}\left\langle \xi ,\xi \right\rangle -\left\langle \nabla
_{\xi }e_{i},\xi \right\rangle \xi \\
&=&\left\langle e_{i},\nabla _{\xi }\xi \right\rangle \xi =\kappa \left(
e_{i}\right) \xi .
\end{eqnarray*}%
Thus,%
\begin{eqnarray*}
D_{Q}\left( \xi \cdot _{M}\nabla _{\xi }^{\Sigma Q}\right)
&=&\sum_{i=1}^{n}e_{i}\cdot _{M}\xi \cdot _{M}\left( \nabla _{\xi }^{\Sigma
Q}\nabla _{e_{i}}^{\Sigma Q}+\kappa \left( e_{i}\right) \nabla _{\xi
}^{\Sigma Q}\right) -\frac{1}{2}\kappa \cdot _{M}\xi \cdot _{M}\nabla _{\xi
}^{\Sigma Q} \\
&=&\sum_{i=1}^{n}e_{i}\cdot _{M}\xi \cdot _{M}\nabla _{\xi }^{\Sigma
Q}\nabla _{e_{i}}^{\Sigma Q}+\sum_{i=1}^{n}\kappa \left( e_{i}\right)
e_{i}\cdot _{M}\xi \cdot _{M}\nabla _{\xi }^{\Sigma Q}-\frac{1}{2}\kappa
\cdot _{M}\xi \cdot _{M}\nabla _{\xi }^{\Sigma Q} \\
&=&\sum_{i=1}^{n}e_{i}\cdot _{M}\xi \cdot _{M}\nabla _{\xi }^{\Sigma
Q}\nabla _{e_{i}}^{\Sigma Q}+\frac{1}{2}\kappa \cdot _{M}\xi \cdot
_{M}\nabla _{\xi }^{\Sigma Q} \\
&=&-\left( \xi \cdot _{M}\nabla _{\xi }^{\Sigma Q}\right) \left( e_{i}\cdot
_{Q}\nabla _{e_{i}}^{\Sigma Q}\right) -\left( \xi \cdot _{M}\nabla _{\xi
}^{\Sigma Q}\right) \frac{1}{2}\kappa \cdot _{M} \\
&=&-\left( \xi \cdot _{M}\nabla _{\xi }^{\Sigma Q}\right) D_{Q}-\left( \xi
\cdot _{M}\nabla _{\xi }^{\Sigma Q}\right) \kappa \cdot _{M} \\
&=&-\left( \xi \cdot _{M}\nabla _{\xi }^{\Sigma Q}\right) D_{Q}+\kappa \cdot
_{M}\left( \xi \cdot _{M}\nabla _{\xi }^{\Sigma Q}\right) .
\end{eqnarray*}
\end{proof}

\begin{corollary}
\label{minimalCaseDFeigenvSpectrumCor}If $\kappa =0$, then the spectrum of $%
D_{\mathcal{F}}$ contains a countable number of real eigenvalues, and there
exists a complete orthonormal basis of $L^{2}\left( \Sigma Q\right) $
consisting of smooth eigensections of $D_{\mathcal{F}}$ .

\begin{proof}
If $\kappa =0$, we consider the essentially self-adjoint, elliptic operator 
\begin{equation*}
L=D_{Q}+D_{\mathcal{F}}.
\end{equation*}%
There exists a complete orthonormal basis of $L^{2}\left( \Sigma Q\right) $
consisting of smooth eigensections of $L$, and each eigenspace is
finite-dimensional. By Lemma \ref{CommutatorLemma}, $D_{Q}D_{\mathcal{F}}+D_{%
\mathcal{F}}D_{Q}=0$, so $L^{2}=D_{Q}^{2}+D_{\mathcal{F}}^{2}$, and $D_{%
\mathcal{F}}$ commutes with $L^{2}$. Then $D_{\mathcal{F}}$ restricts to a
self-adjoint operator on the finite-dimensional eigenspaces of $L^{2}$ and
thus has pure real eigenvalue spectrum restricted to those subspaces. The
result follows.
\end{proof}
\end{corollary}

\begin{remark}
As shown in Example \ref{T2Example}, it is possible that the spectrum of $D_{%
\mathcal{F}}$ is $\mathbb{R}$ but also contains a countable number of real
eigenvalues, whose smooth eigensections form a complete orthonormal basis of 
$L^{2}\left( \Sigma Q\right) $.
\end{remark}

\begin{remark}
Suppose instead that $\kappa =df$. Note that this means that $f$ is a basic
function, since otherwise $\kappa $ would have $\xi ^{\ast }$ components.
Then we modify the metric on $M$ so that $\left\langle \xi ,\xi
\right\rangle ^{\prime }=e^{2f}$ but otherwise keep everything the same.
Then the leafwise volume form is 
\begin{equation*}
\chi ^{\prime }=e^{f}\xi ^{\ast },
\end{equation*}%
and 
\begin{equation*}
d\chi ^{\prime }=-\left( \kappa -df\right) \wedge \chi +\varphi _{0}=\varphi
_{0},
\end{equation*}%
so that $\kappa ^{\prime }=0$. Then in the new metric $L^{\prime
}=D_{Q}^{\prime }+D_{\mathcal{F}}^{\prime }$ has the same properties, and $%
D_{\mathcal{F}}^{\prime }$ commutes with $\left( L^{\prime }\right) ^{2}$.
But observe that $D_{\mathcal{F}}^{\prime }=e^{-f}D_{\mathcal{F}}$ because
for all $\psi \in \Gamma \left( \Sigma Q\right) $, $\xi ^{\prime }\cdot
_{M}^{\prime }\psi =\xi \cdot _{M}\psi $, and $\xi ^{\prime }=e^{-f}\xi $.
In examples it appears that $D_{\mathcal{F}}$ does not have a complete basis
of eigenvectors, even though $D_{\mathcal{F}}^{\prime }$ does.
\end{remark}

\section{Adiabatic limits}

In this section, given the bundle-like metric $g$ on $\left( M,\mathcal{F}%
\right) $, we consider the family of metrics%
\begin{equation*}
g_{f}=f^{2}\xi ^{\ast }\otimes \xi ^{\ast }+g_{\xi ^{\bot }},
\end{equation*}%
where $f$ is a positive basic function on $M$. This metric is bundle-like
for the foliation and has the same transverse metric as the original metric,
and $\xi _{f}=\frac{1}{f}\xi $ is the corresponding unit tangent vector
field of the foliation.

\begin{lemma}
The spaces $L^{2}\left( \Gamma _{b}\left( \Sigma Q\right) \right) $ and $%
L^{2}\left( \Gamma _{b}\left( \Sigma Q\right) \right) ^{\bot }$ are the same
for any such metric $g_{f}$.
\end{lemma}

\begin{proof}
The space $\Gamma _{b}\left( \Sigma Q\right) $ does not depend on the metric
and thus is independent of $f$. Since $f$ is a smooth positive function, we
see easily that $L^{2}\left( \Gamma _{b}\left( \Sigma Q\right) \right) $ is
also independent of $f$. Next, suppose that $\alpha $ is orthogonal to any
given $\beta \in \Gamma _{b}\left( \Sigma Q\right) $ with respect to the old
metric. Then if we let $\left( \bullet ,\bullet \right) $ denote the
original pointwise metric on $\Sigma Q$, we have that $\left( \alpha ,\beta
\right) $ is independent of $f$ since $\beta $ has no components with $\xi
^{\ast }$. Also, $\nu \wedge \xi ^{\ast }$ is the original volume form on $M$
with $\nu $ the transverse volume form. In the new metric, $f\nu \wedge \xi
^{\ast }$ is the volume form. Then 
\begin{equation*}
\left\langle \alpha ,\beta \right\rangle _{f}=\int \left( \alpha ,\beta
\right) f\nu \wedge \xi ^{\ast }=\int \left( \alpha ,f\beta \right) \nu
\wedge \xi ^{\ast }=0
\end{equation*}%
since $f\beta $ is also a basic form. Therefore, we also have that the space 
$L^{2}\left( \Gamma _{b}\left( \Sigma Q\right) \right) ^{\bot }$ is
independent of $f$.
\end{proof}

Recall that the basic component $\kappa _{b}$ of the mean curvature form $%
\kappa $ is always a closed form and defines a class $\left[ \kappa _{b}%
\right] $ in basic cohomology $H_{b}^{1}\left( M,\mathcal{F}\right) $ that
is invariant of the transverse Riemannian foliation structure and
bundle-like metric (see \cite{AL}). Such a Riemannian foliation is taut if
and only if $\left[ \kappa _{b}\right] =0$. Also, recall from \cite{Dom}:
given any Riemannian foliation $\left( M,\mathcal{F}\right) $ with
bundle-like metric, there exists another bundle-like metric on $M$ with
identical transverse metric such that the mean curvature is basic.

\begin{theorem}
\label{CollapsingThm}Let $M$ be a closed Riemannian spin manifold, endowed
with an oriented Riemannian flow given by the unit vector field $\xi $.
Suppose that the mean curvature form $\kappa $ is basic. Let $D_{M,f}$ be
the Dirac operator associated to the metric $g_{f}$ and spin structure. The
eigenvalues of $D_{M,f}$ are $\left\{ \lambda _{j}\left( f\right) \right\}
_{j=1}^{\infty }\cup \left\{ \mu _{k}\left( f\right) \right\} _{k=1}^{\infty
}$, corresponding to the restrictions of $D_{M,f}$ to $L^{2}\left( \Gamma
_{b}\left( \Sigma Q\right) \right) $ and $L^{2}\left( \Gamma _{b}\left(
\Sigma Q\right) \right) ^{\bot }$, respectively. Then these eigenvalues can
be indexed such that

\begin{enumerate}
\item 
\begin{enumerate}
\item ($n$ even) as $f\rightarrow 0$, $\lambda _{j}\left( f\right) $
converges to eigenvalues of the basic Dirac operator $D_{b}$.

\item ($n$ odd) as $f\rightarrow 0$, $\lambda _{j}\left( f\right) $
converges to the eigenvalues of the basic Dirac operators $\pm D_{b}$.
\end{enumerate}

In the cases above, the convergence is uniform in $j$.

\item If $\mathcal{F}$ is taut (i.e. $\kappa =dh$ for a function $h$), the
nonzero eigenvalues in $\left\{ \mu _{k}\left( f\right) \right\} $ approach $%
\pm \infty $ as $f\rightarrow 0$ uniformly with $\frac{df}{f}$ uniformly
bounded.
\end{enumerate}
\end{theorem}

\begin{proof}
(1a) Observe that $\xi _{f}=\frac{1}{f}\xi ,$ $\xi _{f}^{\ast }=f\xi ,$ $%
\kappa _{f}=\kappa -\frac{df}{f}$ and $\Omega _{f}=f\Omega $. For the case
where $n$ is even, from (\ref{D_mD_tr formula}),%
\begin{eqnarray*}
D_{M,f} &=&D_{Q,f}-\frac{1}{2}\xi _{f}\cdot _{M,f}\Omega _{f}\cdot
_{M,f}+\xi _{f}\cdot _{M,f}\nabla _{\xi _{f}}^{\Sigma Q} \\
&=&D_{Q,f}-\frac{f}{2}\xi _{f}\cdot _{M,f}\Omega \cdot _{M,f}+\frac{1}{f}\xi
_{f}\cdot _{M,f}\nabla _{\xi }^{\Sigma Q}.
\end{eqnarray*}%
Then for any basic spinor $\psi $,%
\begin{equation}
D_{M,f}\left( \psi \right) =D_{b,f}\psi -\frac{f}{2}\xi _{f}\cdot
_{M,f}\Omega \cdot _{M}\psi ,
\end{equation}
since $\kappa $ is basic. Thus,

\begin{eqnarray*}
\frac{\left\Vert \left( D_{M,f}-D_{b,f}\right) \psi \right\Vert _{L^{2}}}{%
\left\Vert \psi \right\Vert _{L^{2}}} &=&\frac{\left\Vert \frac{f}{2}\xi
_{f}\cdot _{M,f}\Omega \cdot _{M}\psi \right\Vert _{L^{2}}}{\left\Vert \psi
\right\Vert _{L^{2}}}\leq \frac{\left\Vert \frac{f}{2}\Omega \cdot _{M}\psi
\right\Vert _{L^{2}}}{\left\Vert \psi \right\Vert _{L^{2}}} \\
&\leq &\frac{\max \left\vert f\right\vert }{2}C,
\end{eqnarray*}%
where $C$ is the operator norm of $\left( \Omega \cdot _{M}\right) $. Thus 
\begin{equation*}
\frac{\left\Vert \left( D_{M,f}-D_{b,f}\right) \psi \right\Vert _{L^{2}}}{%
\left\Vert \psi \right\Vert _{L^{2}}}\rightarrow 0
\end{equation*}%
uniformly in $f$ and $\psi $, hence 
\begin{equation*}
\left\Vert \left( D_{M,f}-D_{b,f}\right) \right\Vert _{Op}\leq \frac{\max
\left\vert f\right\vert }{2}C\rightarrow 0
\end{equation*}%
as $f\rightarrow 0$ uniformly. Since the eigenvalues of $D_{b,f}$ are
constant in $f$ and are those of $D_{b}$(see \cite{HabRic}), the eigenvalues
of $D_{M,f}$ converge to those of $D_{b}$, because the spectrum is
continuous as a function of the operator norm (see Lemma \ref%
{OperatorTheoryLemma} in the appendix).

(1b) For the case where $n$ is odd, from (\ref{D_mD_tr formula}),%
\begin{eqnarray*}
D_{M,f} &=&\xi \cdot _{M,f}\left( D_{Q,f}\oplus \left( -D_{Q,f}\right)
\right) -\frac{1}{2}\xi _{f}\cdot _{M,f}\Omega _{f}\cdot _{M,f}+\xi
_{f}\cdot _{M,f}\nabla _{\xi _{f}}^{\Sigma Q\oplus \Sigma Q} \\
&=&\xi \cdot _{M,f}\left( D_{Q,f}\oplus \left( -D_{Q,f}\right) \right) -%
\frac{f}{2}\xi _{f}\cdot _{M,f}\Omega \cdot _{M,f}+\frac{1}{f}\xi _{f}\cdot
_{M,f}\nabla _{\xi }^{\Sigma Q\oplus \Sigma Q}.
\end{eqnarray*}%
Then, since $\kappa $ is basic, for any basic spinor $\psi =\left( \psi
_{1},\psi _{2}\right) \in \Gamma _{b}\left( \Sigma Q\oplus \Sigma Q\right) $,%
\begin{eqnarray*}
D_{M,f}\left( \psi \right) &=&\xi \cdot _{M,f}\left( D_{Q,f}\psi _{1}\oplus
\left( -D_{Q,f}\psi _{2}\right) \right) -\frac{f}{2}\xi _{f}\cdot
_{M,f}\Omega \cdot _{M,f}\left( \psi _{1},\psi _{2}\right) +\frac{1}{f}\xi
_{f}\cdot _{M,f}\nabla _{\xi }^{\Sigma Q}\left( \psi _{1},\psi _{2}\right) \\
&=&\xi \cdot _{M,f}\left( D_{b,f}\psi _{1},-D_{b,f}\psi _{2}\right) -\frac{f%
}{2}\xi _{f}\cdot _{M,f}\Omega \cdot _{M}\left( \psi _{1},\psi _{2}\right) .
\end{eqnarray*}

Thus,

\begin{eqnarray*}
\frac{\left\Vert D_{M,f}\left( \psi _{1},\psi _{2}\right) -\xi \cdot
_{M,f}\left( D_{b,f}\psi _{1},-D_{b,f}\psi _{2}\right) \right\Vert _{L^{2}}}{%
\left\Vert \psi \right\Vert _{L^{2}}} &=&\frac{\left\Vert \frac{f}{2}\xi
_{f}\cdot _{M,f}\Omega \cdot _{M}\psi \right\Vert _{L^{2}}}{\left\Vert \psi
\right\Vert _{L^{2}}}\leq \frac{\left\Vert \frac{f}{2}\Omega \cdot _{M}\psi
\right\Vert _{L^{2}}}{\left\Vert \psi \right\Vert _{L^{2}}} \\
&\leq &\frac{\max \left\vert f\right\vert }{2}C,
\end{eqnarray*}%
where $C$ is the operator norm of $\left( \Omega \cdot _{M}\right) $. The
same conclusions follow.

($2^{\prime }$) Now we suppose the particular case that $\kappa =0$. Then $%
\kappa _{f}=-\frac{df}{f}$. For the case where $n$ is even, 
\begin{eqnarray*}
D_{M,f} &=&D_{Q,f}-\frac{f}{2}\xi \cdot _{M}\Omega \cdot _{M}+\frac{1}{f}\xi
\cdot _{M}\nabla _{\xi }^{\Sigma Q} \\
&=&\sum_{i=1}^{n}e_{i}\cdot _{Q}\nabla _{e_{i}}^{\Sigma Q}+\frac{1}{2}\left( 
\frac{df}{f}\right) \cdot _{Q}-\frac{f}{2}\xi \cdot _{M}\Omega \cdot _{M}+%
\frac{1}{f}D_{\mathcal{F}}.
\end{eqnarray*}

We consider the elliptic operator $L_{f}=D_{\func{tr}}+\frac{1}{f}D_{%
\mathcal{F}}=D_{Q}+\frac{1}{f}D_{\mathcal{F}}$, which is self-adjoint with
respect to the original metric and therefore has discrete real spectrum.
Then if $\ast $ is used as the adjoint with respect to the $L^{2}\left(
M,g_{f}\right) $ metric,

\begin{eqnarray*}
L_{f}^{\ast }L_{f} &=&\left( D_{\func{tr}}+\frac{1}{f}D_{\mathcal{F}}\right)
^{\ast }\left( D_{\func{tr}}+\frac{1}{f}D_{\mathcal{F}}\right)  \\
&=&\left( D_{\func{tr}}^{\ast }+\frac{1}{f}D_{\mathcal{F}}^{\ast }\right)
\left( D_{\func{tr}}+\frac{1}{f}D_{\mathcal{F}}\right)  \\
&=&\left( D_{\func{tr}}+\frac{df}{f}\cdot _{Q}+\frac{1}{f}D_{\mathcal{F}%
}\right) \left( D_{\func{tr}}+\frac{1}{f}D_{\mathcal{F}}\right)  \\
&=&D_{\func{tr}}^{2}+\frac{df}{f}\cdot _{Q}D_{\func{tr}}+\frac{1}{f}D_{%
\mathcal{F}}D_{\func{tr}}+D_{\func{tr}}\circ \frac{1}{f}D_{\mathcal{F}}+%
\frac{df}{f^{2}}\cdot _{Q}D_{\mathcal{F}}+\frac{1}{f^{2}}D_{\mathcal{F}}^{2}
\\
&=&D_{\func{tr}}^{2}+\frac{df}{f}\cdot _{Q}D_{\func{tr}}+\frac{1}{f}D_{%
\mathcal{F}}D_{\func{tr}}-\frac{df}{f^{2}}\cdot _{Q}D_{\mathcal{F}}+\frac{1}{%
f}D_{\func{tr}}D_{\mathcal{F}}+\frac{df}{f^{2}}\cdot _{Q}D_{\mathcal{F}}+%
\frac{1}{f^{2}}D_{\mathcal{F}}^{2} \\
&=&D_{\func{tr}}^{2}+\frac{df}{f}\cdot _{Q}D_{\func{tr}}+\frac{1}{f^{2}}D_{%
\mathcal{F}}^{2}
\end{eqnarray*}%
where $D_{\func{tr}}=\sum_{i=1}^{n}e_{i}\cdot _{Q}\nabla _{e_{i}}^{\Sigma Q}$%
, which is self-adjoint with respect to the original metric. Clearly $%
L_{f}^{\ast }L_{f}$ is nonnegative, elliptic, and self-adjoint with respect
to the new metric and thus has discrete spectrum. The operator$\ D_{\mathcal{%
F}}$ restricts to the eigenspaces of $L_{f}^{\ast }L_{f}$ since they
commute. Indeed, $D_{\mathcal{F}}$ anticommutes with $D_{\func{tr}}$ and
with $\frac{df}{f^{2}}\cdot _{Q}$ and commutes with $\frac{1}{f}$. By
Corollary \ref{minimalCaseDFeigenvSpectrumCor}, we may restrict to an
eigenspace of $D_{\mathcal{F}}$ corresponding to an eigenvalue $\alpha \neq 0
$ (since we are only considering antibasic sections now), and we see that
such an eigenvalue, normalized antibasic eigensection pair $\lambda
_{f},\psi _{f}$ satisfies%
\begin{eqnarray*}
\left\langle L_{f}^{\ast }L_{f}\psi _{f},\psi _{f}\right\rangle _{f}
&=&\left\langle \left( D_{\func{tr}}^{2}+\frac{df}{f}\cdot _{Q}D_{\func{tr}}+%
\frac{1}{f^{2}}D_{\mathcal{F}}^{2}\right) \psi _{f},\psi _{f}\right\rangle
_{f} \\
&=&\left\langle \left( D_{\func{tr}}^{2}+\frac{df}{f}\cdot _{Q}D_{\func{tr}}+%
\frac{1}{f^{2}}\alpha ^{2}\right) \psi _{f},\psi _{f}\right\rangle _{f} \\
&=&\left\langle \left( \frac{1}{f}D_{\func{tr}}\left( fD_{\func{tr}}\right) +%
\frac{1}{f^{2}}\alpha ^{2}\right) \psi _{f},\psi _{f}\right\rangle _{f} \\
&=&\left\langle D_{\func{tr}}\left( fD_{\func{tr}}\right) \psi _{f},\psi
_{f}\right\rangle +\left\langle \frac{1}{f^{2}}\alpha ^{2}\psi _{f},\psi
_{f}\right\rangle _{f} \\
&=&\left\langle D_{\func{tr}}\psi _{f},D_{\func{tr}}\psi _{f}\right\rangle
_{f}+\alpha ^{2}\left\langle \frac{1}{f^{2}}\psi _{f},\psi _{f}\right\rangle
_{f} \\
&\geq &\frac{\alpha ^{2}}{\max \left( f^{2}\right) }\rightarrow \infty 
\end{eqnarray*}%
as $f\rightarrow 0$ uniformly. Thus, the eigenvalues of $L_{f}^{\ast }L_{f}$
go to $+\infty $ as $f\rightarrow 0$ uniformly. Since the eigenvalues of $%
L_{f}^{\ast }L_{f}$ are precisely the squares of the eigenvalues of $L_{f}$,
we also get that the eigenvalues of $L_{j}$ approach $\pm \infty $ as $%
f\rightarrow 0$ uniformly. Next, observe that 
\begin{equation*}
\left\Vert D_{M,f}-L_{f}\right\Vert _{Op}=\left\Vert \frac{1}{2}\left( \frac{%
df}{f}\right) \cdot _{Q}-\frac{f}{2}\xi \cdot _{M}\Omega \cdot
_{M}\right\Vert _{Op}\leq \frac{1}{2}\max \left\vert f\right\vert \max
\left\Vert \Omega \right\Vert +\frac{1}{2}\max \left\vert \frac{df}{f}%
\right\vert \text{,}
\end{equation*}%
and the right hand side remains bounded as $f\rightarrow 0$ uniformly with $%
\frac{\left\vert df\right\vert }{f}$ bounded. Thus, since the spectrum is
continuous as a function of the operator norm (see Lemma \ref%
{OperatorTheoryLemma}), the eigenvalues of $D_{M,f}$ go to $\pm \infty $ as $%
f\rightarrow 0$ uniformly with $\frac{\left\vert df\right\vert }{f}$
bounded. The $n$ odd case is similar.

(2) Now, suppose that $\kappa $ is an exact form, so that $\kappa =dh$ for
some function $h$ (which must be basic; otherwise $\kappa $ would have a $%
\xi ^{\ast }$ component). Then we may multiply the leafwise metric by $%
\widetilde{f}^{2}$ where $\widetilde{f}=\exp \left( h\right) $, and then in
the new metric $\widetilde{\kappa }=0$. Then, given any positive function $f$%
, $g_{f}=f^{2}\xi ^{\ast }\otimes \xi ^{\ast }+g_{\xi ^{\bot }}=\left( f%
\widetilde{f}^{-1}\right) ^{2}\widetilde{f}^{2}\xi ^{\ast }\otimes \xi
^{\ast }+g_{\xi ^{\bot }}$. Suppose that $f\rightarrow 0$ uniformly with $%
\frac{df}{f}$ uniformly bounded; then $f\widetilde{f}^{-1}\rightarrow 0$
uniformly and $\frac{d\left( f\widetilde{f}^{-1}\right) }{f\widetilde{f}^{-1}%
}=\frac{d\left( f\right) }{f}-\frac{d\left( \widetilde{f}\right) }{%
\widetilde{f}}$ is also uniformly bounded. By the result in ($2^{\prime }$)
above, the nonzero eigenvalues in $\left\{ \mu _{k}\left( f\right) \right\} $
approach $\pm \infty $.
\end{proof}

\begin{remark}
Example \ref{CarExample} shows that in the case that $\mathcal{F}$ is not
taut, the methods of the proof for part (2) do not work. In this example,
the only eigenvalue of $D_{\mathcal{F}}$ is $0$, corresponding to the basic
sections, and yet the spectrum of $D_{\mathcal{F}}$ is $\mathbb{R}$. So the
conclusion of Corollary \ref{minimalCaseDFeigenvSpectrumCor} does not hold
even though $\kappa $ is basic. We conjecture that the conclusion (2)\ is
false for general Riemannian foliations.
\end{remark}

\section{Examples\label{examplesSection}}

\begin{example}
\label{T2Example}Consider $M=T^{2}=\mathbb{R}^{2}\diagup \left( 2\pi \mathbb{%
Z}\right) ^{2}$, the Euclidean two-dimensional torus, with a constant linear
flow $\xi =a\partial _{x}+b\partial _{y}$, where $a^{2}+b^{2}=1$. The spinor
bundle $\Sigma M$ is $\mathbb{C}^{2}\times M$, and we consider the Clifford
multiplication $\left( c\partial _{x}+d\partial _{y}\right) =\left( 
\begin{array}{cc}
0 & -c+di \\ 
c+di & 0%
\end{array}%
\right) $. The bundle $Q=\xi ^{\bot }=span\left\{ -b\partial _{x}+a\partial
_{y}\right\} $, and $\Sigma Q=\mathbb{C}\times M$. Covariant derivatives are
the same as directional derivatives. The standard metric is $g=dx^{2}+dy^{2}$%
, and we consider the perturbed metric 
\begin{equation*}
g_{f}=g_{t}=dx^{2}+dy^{2}+(t^{2}-1)\left( \xi ^{\ast }\right)
^{2}=dx^{2}+dy^{2}+(t^{2}-1)\left( a~dx+b~dy\right) ^{2}
\end{equation*}%
with $f\left( t\right) =t$. Since the foliation for this and the original
metric is totally geodesic, $\nabla ^{\Sigma M}=\nabla ^{\Sigma Q\oplus
\Sigma Q}$. Then 
\begin{eqnarray*}
D_{M,t} &=&\sum e_{j}\cdot _{M}\nabla _{e_{j}}^{\Sigma M}+\xi _{t}\cdot
_{t}\nabla _{\xi _{t}}^{\Sigma M} \\
&=&\sum e_{j}\cdot _{M}\nabla _{e_{j}}^{\Sigma M}+\frac{1}{t}\xi \cdot
_{M}\nabla _{\xi }^{\Sigma M} \\
&=&D_{M}+\left( \frac{1}{t}-1\right) \xi \cdot _{M}\nabla _{\xi }^{M}.
\end{eqnarray*}%
We now compute the eigenvalues of $D_{M,t}$. Observe that 
\begin{equation*}
\xi \cdot _{M}\nabla _{\xi }^{M}=\left( 
\begin{array}{cc}
0 & -a+bi \\ 
a+bi & 0%
\end{array}%
\right) \left( a\partial _{x}+b\partial _{y}\right) .
\end{equation*}%
Consider the space $V_{m,n}=\left\{ \left( 
\begin{array}{c}
c \\ 
d%
\end{array}%
\right) \exp \left( i\left( mx+ny\right) \right) :c,d\in \mathbb{C}\right\} $
, so that the Hilbert sum$\bigoplus\limits_{m,n\in \mathbb{Z}%
}V_{m,n}=L^{2}\left( \Sigma M\right) $.
We see that%
\begin{eqnarray*}
&&D_{M,t}\left( \left( 
\begin{array}{c}
c \\ 
d%
\end{array}%
\right) \exp \left( i\left( mx+ny\right) \right) \right) \\
&=&\left( \left( 
\begin{array}{cc}
0 & -\partial _{x}+i\partial _{y} \\ 
\partial _{x}+i\partial _{y} & 0%
\end{array}%
\right) +\left( \frac{1}{t}-1\right) \left( 
\begin{array}{cc}
0 & -a+bi \\ 
a+bi & 0%
\end{array}%
\right) \left( a\partial _{x}+b\partial _{y}\right) \right) \left( \left( 
\begin{array}{c}
c \\ 
d%
\end{array}%
\right) \exp \left( i\left( mx+ny\right) \right) \right) \\
&=&\left( \left( 
\begin{array}{cc}
0 & -im-n \\ 
im-n & 0%
\end{array}%
\right) +\left( \frac{1}{t}-1\right) \left( 
\begin{array}{cc}
0 & -a+bi \\ 
a+bi & 0%
\end{array}%
\right) \left( iam+ibn\right) \right) \left( \left( 
\begin{array}{c}
c \\ 
d%
\end{array}%
\right) \exp \left( i\left( mx+ny\right) \right) \right) .
\end{eqnarray*}%
The matrix is%
\begin{equation*}
\left( 
\begin{array}{cc}
0 & -im-n+\frac{1}{t}\left( -t+1\right) \left( -a+ib\right) \left(
iam+ibn\right) \\ 
im-n+\frac{1}{t}\left( -t+1\right) \left( a+ib\right) \left( iam+ibn\right)
& 0%
\end{array}%
\right) \allowbreak .
\end{equation*}
The eigenvalues are $\pm \sqrt{q}$, where%
\begin{equation*}
q=m^{2}+n^{2}-\left( am+bn\right) ^{2}+\frac{1}{t^{2}}\left( am+bn\right)
^{2}.
\end{equation*}%
\newline
So, in the case where $\frac{b}{a}$ is rational, the set of basic sections
of $\Sigma M$ is 
\begin{equation*}
\left\{ \left( 
\begin{array}{c}
c \\ 
d%
\end{array}%
\right) \exp \left( i\left( mx+ny\right) \right) :c,d\in \mathbb{C},m,n\in 
\mathbb{Z},am+bn=0\right\} .
\end{equation*}
Also, $\Sigma Q=\mathbb{C}$, and the basic Dirac operator is $D_{b}=i\theta $%
, where $\theta \bot \xi $. It has eigenvalues 
\begin{equation*}
\left\{ m\sqrt{1+\frac{b^{2}}{a^{2}}}:m\in \mathbb{Z}\right\}
\end{equation*}
with eigensections of the form $\left\{ \exp \left( i\left( mx+ny\right)
\right) :m,n\in \mathbb{Z},am+bn=0\right\} $. Actually, $M\diagup \mathcal{F}
$ is a circle of radius $\frac{2\pi }{\sqrt{1+\frac{b^{2}}{a^{2}}}}$. As can
be seen above, the eigenvalues $D_{M,t}$ are 
\begin{equation*}
\pm \sqrt{m^{2}+n^{2}-\left( am+bn\right) ^{2}+\frac{1}{t^{2}}\left(
am+bn\right) ^{2}}
\end{equation*}
with $m,n\in \mathbb{Z}$. The eigenvalues with $am+bn=0$ are independent of $%
t$ and trivially converge to the eigenvalues of $D_{b}\oplus $ $-D_{b}$. All
other eigenvalues go to $\pm \infty $ as $t\rightarrow 0$.

On the other hand, if $\frac{b}{a}$ is irrational, the basic sections of $%
\Sigma M$ are $\left\{ \left( 
\begin{array}{c}
c \\ 
d%
\end{array}%
\right) :c,d\in \mathbb{C}\right\} $, since each leaf is dense. The basic
Dirac operator is the zero operator and only has the eigenvalue $0$. Also,
since $am+bn\neq 0$ for all $\left( m,n\right) \in \mathbb{Z}^{2}\setminus
\left\{ \left( 0,0\right) \right\} $, the expression above implies that
every eigenvalue besides $0$ goes to $\pm \infty $ as $t\rightarrow 0$.

These results are consistent with our theorem. We also find the spectrum of
the operator 
\begin{equation*}
\xi \cdot _{M}\nabla _{\xi }^{M}=\left( 
\begin{array}{cc}
0 & -a+bi \\ 
a+bi & 0%
\end{array}%
\right) \left( a\partial _{x}+b\partial _{y}\right) .
\end{equation*}%
Applied to an element of $\left\{ \left( 
\begin{array}{c}
c \\ 
d%
\end{array}%
\right) \exp \left( i\left( mx+ny\right) \right) :c,d\in \mathbb{C},m,n\in 
\mathbb{Z},am+bn=0\right\} $, we get%
\begin{equation*}
\left( \xi \cdot _{M}\nabla _{\xi }^{M}\right) \left( 
\begin{array}{c}
c \\ 
d%
\end{array}%
\right) \exp \left( i\left( mx+ny\right) \right) =\left( iam+ibn\right)
\left( 
\begin{array}{cc}
0 & -a+bi \\ 
a+bi & 0%
\end{array}%
\right) \left( 
\begin{array}{c}
c \\ 
d%
\end{array}%
\right) \exp \left( i\left( mx+ny\right) \right) ,
\end{equation*}%
and the matrix restricted to this subspace is%
\begin{equation*}
\left( iam+ibn\right) \left( 
\begin{array}{cc}
0 & -a+bi \\ 
a+bi & 0%
\end{array}%
\right) =\left( 
\begin{array}{cc}
0 & \left( -a+ib\right) \left( iam+ibn\right) \\ 
\left( a+ib\right) \left( iam+ibn\right) & 0%
\end{array}%
\right) \allowbreak .
\end{equation*}%
The eigenvalues are obviously $\pm \left( am+bn\right) $, so that in the
irrational slope case $0$ is a limit point of the eigenvalues of $D_{%
\mathcal{F}}$. In fact, the eigenvalues are dense in $\mathbb{R}$. Note that
the whole spectrum is $\mathbb{R}$ because it is closed, even though there
exists an orthonormal basis of $L^{2}\left( \Sigma M\right) $ consisting of
eigensections. The problem is that $\left( D_{\mathcal{F}}-\lambda I\right)
^{-1}$ for any $\lambda $ not in the spectrum, but this operator is not a
bounded operator.
\end{example}

\begin{example}
Consider $M=T^{3}=\mathbb{R}^{3}\diagup \left( 2\pi \mathbb{Z}\right) ^{3}$,
the Euclidean $3$-torus, with a constant linear flow $\xi =a\partial
_{x}+b\partial _{y}+c\partial _{z}$, where $a^{2}+b^{2}+c^{2}=1$. The spinor
bundle $\Sigma M$ is $\mathbb{C}^{2}\times M$, and we consider the Clifford
multiplication $\left( c\partial _{x}+d\partial _{y}+e\partial _{z}\right)
=\left( 
\begin{array}{cc}
ie & -c+di \\ 
c+di & -ie%
\end{array}%
\right) $. The bundle $Q=\xi ^{\bot }=span\left\{ -b\partial _{x}+a\partial
_{y},-c\partial _{x}+a\partial _{z}\right\} $, and $\Sigma Q=\mathbb{C}%
^{2}\times M$. Covariant derivatives are the same as directional
derivatives. The standard metric is $g=dx^{2}+dy^{2}+dz^{2}$, and we
consider the perturbed metric 
\begin{equation*}
g_{f}=g_{t}=g+(t^{2}-1)\left( \xi ^{\ast }\right)
^{2}=dx^{2}+dy^{2}+(t^{2}-1)\left( a~dx+b~dy+cdz\right) ^{2}
\end{equation*}%
with $f\left( t\right) =t$. Since the foliation for this and the original
metric is totally geodesic, $\nabla ^{\Sigma M}=\nabla ^{\Sigma Q}$. Then 
\begin{eqnarray*}
D_{M,t} &=&\sum e_{j}\cdot _{M}\nabla _{e_{j}}^{\Sigma M}+\xi _{t}\cdot
_{t}\nabla _{\xi _{t}}^{\Sigma M} \\
&=&\sum e_{j}\cdot _{M}\nabla _{e_{j}}^{\Sigma M}+\frac{1}{t}\xi \cdot
_{M}\nabla _{\xi }^{\Sigma M} \\
&=&D_{M}+\left( \frac{1}{t}-1\right) \xi \cdot _{M}\nabla _{\xi }^{M}.
\end{eqnarray*}%
We now compute the eigenvalues of $D_{M,t}$. Observe that 
\begin{equation*}
\xi \cdot _{M}\nabla _{\xi }^{M}=\left( 
\begin{array}{cc}
ic & -a+bi \\ 
a+bi & -ic%
\end{array}%
\right) \left( a\partial _{x}+b\partial _{y}+c\partial _{z}\right) .
\end{equation*}%
Consider the space $V_{m,n,k}=\left\{ \left( 
\begin{array}{c}
r \\ 
s%
\end{array}%
\right) \exp \left( i\left( mx+ny+kz\right) \right) :r,s\in \mathbb{C}%
\right\} $ , so that the Hilbert sum$\bigoplus\limits_{m,n,k\in \mathbb{Z}%
}V_{m,n,k}=L^{2}\left( \Sigma M\right) $.

We see that for $\varphi =\left( 
\begin{array}{c}
r \\ 
s%
\end{array}%
\right) \exp \left( i\left( mx+ny+kz\right) \right) $,%
\begin{eqnarray*}
&&D_{M,t}\varphi  \\
&=&\left( \left( 
\begin{array}{cc}
i\partial _{z} & -\partial _{x}+i\partial _{y} \\ 
\partial _{x}+i\partial _{y} & -i\partial _{z}%
\end{array}%
\right) +\left( \frac{1}{t}-1\right) \left( 
\begin{array}{cc}
ic & -a+bi \\ 
a+bi & -ic%
\end{array}%
\right) \left( a\partial _{x}+b\partial _{y}+c\partial _{z}\right) \right)
\varphi  \\
&=&\left( \left( 
\begin{array}{cc}
-k & -im-n \\ 
im-n & k%
\end{array}%
\right) +\left( \frac{1}{t}-1\right) \left( 
\begin{array}{cc}
ic & -a+bi \\ 
a+bi & -ic%
\end{array}%
\right) \left( iam+ibn+ick\right) \right) \varphi .
\end{eqnarray*}%
One can check that the eigenvalues of $D_{M,t}$ restricted to such sections
are 
\begin{equation*}
\pm \sqrt{k^{2}+n^{2}+m^{2}+\frac{\left( 1-t^{2}\right) }{t^{2}}\left(
am+bn+ck\right) ^{2}}.
\end{equation*}%
As $t\rightarrow 0^{+}$, then $\lambda \approx \pm \frac{1}{t}\left\vert
am+bn+ck\right\vert $ if $am+bn+ck\neq 0$, and $\lambda =\pm \sqrt{%
k^{2}+n^{2}+m^{2}}$ otherwise. So as $t\rightarrow 0^{+}$, if $am+bn+ck=0$
(i.e. basic eigensections of $D_{M}$), then the eigenvalues are $\pm \sqrt{%
k^{2}+n^{2}+m^{2}}$ and do not change with $t$. Otherwise, if $am+bn+ck\neq 0
$, then all the eigenvalues go to $\pm \infty $. This is consistent with our
theorem.
\end{example}

\begin{example}
\label{CarExample}Consider the Carri\`{e}re example from \cite{Car} in the $%
3 $-dimensional case. This foliation is not taut, and we will show that the
spectrum of $D_{\mathcal{F}}$ is all of $\mathbb{R}$ in this case, and its
only eigenvalue is $0$, corresponding to the basic sections. Choose $%
A=\left( 
\begin{array}{cc}
2 & 1 \\ 
1 & 1%
\end{array}%
\right) $ to be a symmetric matrix in $\mathrm{SL}_{2}(\mathbb{Z})$, and let 
$\mathbb{T}^{2}=\mathbb{R}^{2}\diagup \mathbb{Z}^{2}$. Note that the
eigenvalues of $A$ are $\lambda =\frac{3+\sqrt{5}}{2},\frac{1}{\lambda }=%
\frac{3-\sqrt{5}}{2}$ corresponding to normalized eigenvectors 
\begin{eqnarray*}
V_{1}\allowbreak &=&\left( 
\begin{array}{c}
\frac{\frac{1}{2}\sqrt{5}+\frac{1}{2}}{\sqrt{\frac{1}{2}\sqrt{5}+\frac{5}{2}}%
} \\ 
\frac{1}{\sqrt{\frac{1}{2}\sqrt{5}+\frac{5}{2}}}%
\end{array}%
\right) =\allowbreak \left( 
\begin{array}{c}
0.850\,65 \\ 
0.525\,73%
\end{array}%
\right) \allowbreak =:G\partial _{x}+K\partial _{y}, \\
V_{2} &=&\left( 
\begin{array}{c}
\frac{-\frac{1}{2}\sqrt{5}+\frac{1}{2}}{\sqrt{-\frac{1}{2}\sqrt{5}+\frac{5}{2%
}}} \\ 
\frac{1}{\sqrt{-\frac{1}{2}\sqrt{5}+\frac{5}{2}}}%
\end{array}%
\right) =\allowbreak \left( 
\begin{array}{c}
-0.525\,73 \\ 
0.850\,65%
\end{array}%
\right) \allowbreak =-K\partial _{x}+G\partial _{y},
\end{eqnarray*}%
respectively. Let the hyperbolic torus $M=\mathbb{T}_{A}^{3}$ be the
quotient of $\mathbb{T}^{2}\times \mathbb{R}$ by the equivalence relation
which identifies $(m,t)$ to $(A(m),t+1)$. We may also think of it as $%
\mathbb{T}^{2}\times \left[ 0,1\right] $ with $\left( m,0\right) $
identified with $\left( Am,1\right) $.

We choose the bundle-like metric so that the vectors $V_{1},V_{2},$ $%
\partial _{t}$ form an orthonormal basis at $t=0$ and in general $\lambda
^{t}V_{1}$,$\lambda ^{-t}V_{2},\partial _{t}$ form an orthonormal basis for $%
t\in \left[ 0,1\right] $. Note that at $t=0$, this is the standard flat
metric on the torus. If we use $^{\ast }$ to denote the adjoint/dual with
respect to the $t=0$ metric, the metric is%
\begin{equation*}
g=dt^{2}+\lambda ^{-2t}\left( V_{1}^{\ast }\right) ^{2}+\lambda ^{2t}\left(
V_{2}^{\ast }\right) ^{2}.
\end{equation*}%
We have that the mean curvature of the flow is $\kappa =\kappa _{b}=-\log
\left( \lambda \right) dt$, since $\chi _{\mathcal{F}}=\lambda
^{t}V_{2}^{\ast }$ is the characteristic form and $d\chi _{\mathcal{F}}=\log
\left( \lambda \right) \lambda ^{t}dt\wedge V_{2}^{\ast }=-\kappa \wedge
\chi _{\mathcal{F}}$. We also have that $\varphi _{0}=0$ for this flow. 
\newline
We choose the trivial spin structure, so that the spin bundle is $M\times 
\mathbb{C}^{2}$ with spinor connection, with Clifford multiplication 
\begin{equation*}
c\left( \lambda ^{-t}V_{2}\right) =\left( 
\begin{array}{cc}
i & 0 \\ 
0 & -i%
\end{array}%
\right) ,c\left( \lambda ^{t}V_{1}\right) =\left( 
\begin{array}{cc}
0 & -1 \\ 
1 & 0%
\end{array}%
\right) ,c\left( \partial _{t}\right) =\left( 
\begin{array}{cc}
0 & i \\ 
i & 0%
\end{array}%
\right) .
\end{equation*}%
We need to calculate the covariant derivatives of spinors. We calculate for $%
\xi =e_{0}=\lambda ^{-t}V_{2}$, $e_{1}=\lambda ^{t}V_{1}$, $e_{2}=\partial
_{t}$. 
\begin{eqnarray*}
\left[ \lambda ^{-t}V_{2},\lambda ^{t}V_{1}\right]  &=&\left[ e_{0},e_{1}%
\right] =0 ,\\
\left[ \lambda ^{-t}V_{2},\partial _{t}\right]  &=&\left[ e_{0},e_{2}\right]
=+\left( \log \lambda \right) \lambda ^{-t}V_{2}=\left( \log \lambda \right)
e_{0}, \\
\left[ \lambda ^{t}V_{1},\partial _{t}\right]  &=&\left[ e_{1},e_{2}\right]
=-\left( \log \lambda \right) \lambda ^{t}V_{1}=-\left( \log \lambda \right)
e_{1}.
\end{eqnarray*}%
Then by the Koszul formula, the Christoffel symbols are%
\begin{eqnarray*}
\Gamma _{00}^{2} &=&\left\langle \nabla _{e_{0}}e_{0},e_{2}\right\rangle =%
\frac{1}{2}\left( -\left\langle \left[ e_{0},e_{2}\right] ,e_{0}\right%
\rangle -\left\langle \left[ e_{0},e_{2}\right] ,e_{0}\right\rangle \right)
=-\log \lambda , \\
\Gamma _{12}^{1} &=&-\Gamma _{11}^{2}=-\Gamma _{02}^{0}=-\log \lambda 
\end{eqnarray*}%
similarly.
Now we use the formula 
\begin{equation*}
\nabla _{X}^{\Sigma M}\psi =X\left( \psi \right) +\frac{1}{2}%
\sum_{i<j}\left\langle \nabla _{X}^{M}e_{i},e_{j}\right\rangle e_{i}\cdot
_{M}e_{j}\cdot _{M}\psi .
\end{equation*}%
Then 
\begin{eqnarray*}
\nabla _{\lambda ^{-t}V_{2}}^{\Sigma M}\varphi  &=&\nabla _{e_{0}}^{\Sigma
M}\varphi =s^{-1}\lambda ^{-t}V_{2}\left( \varphi \right) +\frac{1}{2}%
\sum_{i<j}\left\langle \nabla _{e_{0}}^{M}e_{i},e_{j}\right\rangle
e_{i}\cdot _{M}e_{j}\cdot _{M}\varphi  \\
&=&s^{-1}\lambda ^{-t}V_{2}\left( \varphi \right) +\frac{\log \lambda }{2}%
\left( 
\begin{array}{cc}
0 & 1 \\ 
-1 & 0%
\end{array}%
\right) \varphi =s^{-1}\lambda ^{-t}V_{2}\left( \varphi \right) -\frac{\log
\lambda }{2}e_{1}\cdot _{M}\varphi , \\
\nabla _{\lambda ^{t}V_{1}}^{\Sigma M}\varphi  &=&\nabla _{e_{1}}^{\Sigma
M}\varphi =\lambda ^{t}V_{1}\left( \varphi \right) +\frac{1}{2}%
\sum_{i<j}\left\langle \nabla _{e_{1}}^{M}e_{i},e_{j}\right\rangle
e_{i}\cdot _{M}e_{j}\cdot _{M}\varphi  \\
&=&\lambda ^{t}V_{1}\left( \varphi \right) -\frac{\log \lambda }{2}\left( 
\begin{array}{cc}
i & 0 \\ 
0 & -i%
\end{array}%
\right) \varphi =\lambda ^{t}V_{1}\left( \varphi \right) -\frac{\log \lambda 
}{2}e_{0}\cdot _{M}\varphi , \\
\nabla _{\partial _{t}}^{\Sigma M}\varphi  &=&\partial _{t}\varphi .
\end{eqnarray*}%
With $\xi =\lambda ^{-t}V_{2}$, the connection satisfies%
\begin{eqnarray*}
\nabla _{\xi }^{\Sigma M}\varphi  &=&\lambda ^{-t}V_{2}\left( \varphi
\right) +\frac{\log \lambda }{2}\left( 
\begin{array}{cc}
0 & 1 \\ 
-1 & 0%
\end{array}%
\right) \varphi  \\
&=&\nabla _{\xi }^{\Sigma Q}\varphi +\frac{1}{2}\Omega \cdot _{M}\varphi +%
\frac{1}{2}\xi \cdot _{M}\kappa \cdot _{M}\varphi  \\
&=&\nabla _{\xi }^{\Sigma Q}\varphi -\frac{\log \lambda }{2}\left( 
\begin{array}{cc}
i & 0 \\ 
0 & -i%
\end{array}%
\right) \left( 
\begin{array}{cc}
0 & i \\ 
i & 0%
\end{array}%
\right) \varphi  \\
&=&\nabla _{\xi }^{\Sigma Q}\varphi +\frac{\log \lambda }{2}\left( 
\begin{array}{cc}
0 & 1 \\ 
-1 & 0%
\end{array}%
\right) \varphi ,
\end{eqnarray*}%
so%
\begin{equation*}
\nabla _{\xi }^{\Sigma Q}\varphi =s^{-1}\lambda ^{-t}V_{2}\varphi .
\end{equation*}%
Now we compute%
\begin{eqnarray*}
D_{\mathcal{F}}\varphi  &=&\xi \cdot _{M}\nabla _{\xi }^{\Sigma Q}\varphi  \\
&=&\left( 
\begin{array}{cc}
i & 0 \\ 
0 & -i%
\end{array}%
\right) \lambda ^{-t}V_{2}\varphi  \\
&=&\left( 
\begin{array}{cc}
i\lambda ^{-t}V_{2} & 0 \\ 
0 & -i\lambda ^{-t}V_{2}%
\end{array}%
\right) \varphi .
\end{eqnarray*}%
\newline
So to determine the spectrum of $D_{\mathcal{F}}$, we consider $D_{\mathcal{F%
}}-\mu I$ and determine when it has a bounded inverse. We apply this to a
section of the form%
\begin{equation*}
\varphi =\left( 
\begin{array}{c}
a_{bc} \\ 
f_{bc}%
\end{array}%
\right) e^{2\pi i\left( bx+cy\right) }.
\end{equation*}%
Then 
\begin{eqnarray*}
\left( D_{\mathcal{F}}-\mu I\right) \varphi  &=&\left( 
\begin{array}{cc}
i\lambda ^{-t}V_{2}-\mu  & 0 \\ 
0 & -i\lambda ^{-t}V_{2}-\mu 
\end{array}%
\right) \left( 
\begin{array}{c}
a_{bc} \\ 
f_{bc}%
\end{array}%
\right) e^{2\pi i\left( bx+cy\right) } \\
&=&\left( 
\begin{array}{cc}
i\lambda ^{-t}\left( 2\pi i\left( -Kb+Gc\right) \right) -\mu  & 0 \\ 
0 & -i\lambda ^{-t}\left( 2\pi i\left( -Kb+Gc\right) \right) -\mu 
\end{array}%
\right) \varphi  \\
&=&\left( 
\begin{array}{cc}
-\lambda ^{-t}2\pi \left( -Kb+Gc\right) -\mu  & 0 \\ 
0 & \lambda ^{-t}2\pi \left( -Kb+Gc\right) -\mu 
\end{array}%
\right) \varphi .
\end{eqnarray*}%
Suppose that $\mu $ is actually an eigenvalue of $D_{\mathcal{F}}$. Then $%
\varphi $ must satisfy the condition $\varphi \left( t+1,2x+y,x+y\right)
=\varphi \left( t,x,y\right) $, $\mu $ must be constant, and $\mu =\pm
\lambda ^{-t}2\pi \left( -Kb+Gc\right) $. So only $b=c=0$ is possible,
corresponding to the double eigenvalue $0$. The eigensections are exactly
the sections that depend on $t$ alone, the basic sections.\newline
What is in the other part of the spectrum of $D_{\mathcal{F}}$? We have%
\begin{equation*}
\left( D_{\mathcal{F}}-\mu I\right) ^{-1}\varphi =\left( 
\begin{array}{cc}
-p-\mu  & 0 \\ 
0 & p-\mu 
\end{array}%
\right) ^{-1}\varphi =\allowbreak \left( 
\begin{array}{cc}
-\frac{1}{p+\mu } & 0 \\ 
0 & \frac{1}{p-\mu }%
\end{array}%
\right) \varphi 
\end{equation*}%
acting on sections of the form $\varphi $, which exists as long as $p-\mu
\neq 0$ and $p+\mu \neq 0$, where $p=\lambda ^{-t}2\pi \left( -Kb+Gc\right)
=\lambda ^{-t}2\pi \allowbreak \left( -0.525\,73b+0.850\,65c\right) $ takes
on every number in the range%
\begin{equation*}
\lambda ^{-1}2\pi \allowbreak \allowbreak \left(
-0.525\,73b+0.850\,65c\right) \leq p\leq 2\pi \allowbreak \allowbreak \left(
-0.525\,73b+0.850\,65c\right) ;
\end{equation*}%
that is,%
\begin{equation*}
-1.\,\allowbreak 261\,7\allowbreak b+2.\,\allowbreak 041\,5c\leq p\leq
-3.\,\allowbreak 303\,3\allowbreak b+\allowbreak 5.\,\allowbreak 344\,8c.
\end{equation*}%
So $\mu $ is in the spectrum if and only if $\pm \mu $ is in the set where $%
-1.\,\allowbreak 261\,7\allowbreak b+2.\,\allowbreak 041\,5c\leq \mu \leq
-3.\,\allowbreak 303\,3\allowbreak b+\allowbreak 5.\,\allowbreak 344\,8c$
for any integers $b,c$. Thus, every $\mu \in \mathbb{R}$ is in the spectrum.
\end{example}

\section{Appendix}

We include the following well-known result for completeness, although it certainly is 
contained in more general perturbation theory of linear operators in the literature.
\begin{lemma}
\label{OperatorTheoryLemma}Let $A$ and $B$ be two unbounded, essentially
self-adjoint operators with discrete spectrum and the same domain on a
Hilbert space such that the eigenspaces according to each eigenvalue are
finite-dimensional and the eigenvalues approach $\infty $ in absolute value.
If $\left\Vert A-B\right\Vert _{op}\leq \varepsilon $ for some $\varepsilon
>0$ and%
\begin{equation*}
...\leq \lambda _{j-1}\leq \lambda _{j}\leq \lambda _{j+1}\leq ...
\end{equation*}%
with $j\in\mathbb{Z}$ are the eigenvalues of $A$, counted with multiplicities. Then there is a
numbering of the eigenvalues%
\begin{equation*}
...\leq \mu _{j-1}\leq \mu _{j}\leq \mu _{j+1}\leq ...
\end{equation*}%
of $B$ such that 
\begin{equation*}
\left\vert \lambda _{j}-\mu _{j}\right\vert \leq \varepsilon
\end{equation*}%
for all $j$.

\begin{proof}
First, we prove the result for the case of nonnegative operators. Let $A$
and $B$ be nonnegative, satisfy $\left\Vert A-B\right\Vert _{op}\leq
\varepsilon $, and have domain $\mathcal{D}$. For any subspace $S$ of $%
\mathcal{D}$,%
\begin{equation*}
\sup_{\substack{ \alpha \in S  \\ \left\Vert a\right\Vert =1}}\left\Vert
A\alpha \right\Vert \leq \sup_{\substack{ \alpha \in S  \\ \left\Vert
a\right\Vert =1}}\left\Vert \left( A-B\right) \alpha \right\Vert +\left\Vert
B\alpha \right\Vert \leq \varepsilon +\sup_{\substack{ \alpha \in S  \\ %
\left\Vert a\right\Vert =1}}\left\Vert B\alpha \right\Vert ,
\end{equation*}%
so in particular%
\begin{equation*}
\lambda _{k}=\inf_{\substack{ S\subset \mathcal{D}  \\ \dim S=k}}\left( \sup 
_{\substack{ \alpha \in S  \\ \left\Vert a\right\Vert =1}}\left\Vert A\alpha
\right\Vert \right) \leq \varepsilon +\inf_{\substack{ S\subset \mathcal{D} 
\\ \dim S=k}}\left( \sup_{\substack{ \alpha \in S  \\ \left\Vert
a\right\Vert =1}}\left\Vert B\alpha \right\Vert \right) =\varepsilon +\mu
_{k}.
\end{equation*}%
Reversing the roles of $A$ and $B$, we do obtain $\left\vert \lambda
_{k}-\mu _{k}\right\vert \leq \varepsilon $ for the nonnegative case.\newline
Next, for arbitrary operators $A$ and $B$ that satisfy the hypothesis,
consider the nonnegative operators $A^{\prime }=\left\vert A\right\vert +A$, 
$B^{\prime }=\left\vert B\right\vert +B$, so that $\left\Vert A^{\prime
}-B^{\prime }\right\Vert _{op}\leq 2\varepsilon $. The eigenvalues of $%
A^{\prime }$ and $B^{\prime }$ are $\left\vert \lambda _{k}\right\vert
+\lambda _{k}$ and $\left\vert \mu _{k}\right\vert +\mu _{k}$, respectively,
and the previous argument shows that $\left\vert \lambda _{k}-\mu
_{k}\right\vert \leq \varepsilon $ for all nonnegative eigenvalues $\lambda
_{k}$ and $\mu _{k}$ of $A$ and $B$. Similarly, we apply the previous
argument to $A^{\prime \prime }=\left\vert A\right\vert -A$ and $B^{\prime
\prime }=\left\vert B\right\vert -B$ to show that $\left\vert \lambda
_{k}-\mu _{k}\right\vert \leq \varepsilon $ for all negative eigenvalues $%
\lambda _{k}$ and $\mu _{k}$ of $A$ and $B$.
\end{proof}
\end{lemma}


\begin{thebibliography}{99}
\bibitem{AL} J. A. \'{A}lvarez-L\'{o}pez, \emph{The basic component of the
mean curvature of Riemannian foliations}, Ann. Global Anal. Geom. \textbf{10}
(1992), 179-194.

\bibitem{ALK} J. A. \'{A}lvarez-L\'{o}pez and Y. Kordyukov,
\emph{Adiabatic Limits and Spectral Sequences for Riemannian foliations},
Geom. Funct. Anal. \textbf{10} (2000), 977-1027.



\bibitem{Am} B. Ammann, \emph{The Dirac operator on
collapsing $S^1$ bundles}, in
S\'emin. Th\'eor. Spectr. G\'eom., \textbf{16} (1997), Univ. Grenoble I, Saint-Martin-d'H\`eres, 33-42.


\bibitem{Am-Ba} B. Ammann and C. B\"{a}r, \emph{The Dirac operator on
nilmanifolds and collapsing circle bundles}, Ann. Global Anal. Geom. \textbf{%
16} (1998), no. 3, 221-253.

\bibitem{Baer} C. B\"{a}r, \emph{Extrinsic bounds for eigenvalues of the
Dirac operator}, Ann. Glob. Anal. Geom. \textbf{16} (1998), 573-596.

\bibitem{BiC}
J. M. Bismut and J. Cheeger, 
\emph{$\eta$-invariants and their adiabatic limits},
J. Amer. Math. Soc. \textbf{2} (1989), 33-70.

\bibitem{BKR2} J. Br\"{u}ning, F. W. Kamber, and K. Richardson, \emph{Index
theory for basic Dirac operators on Riemannian foliations}, in \emph{%
Noncommutative geometry and global analysis}, Contemp. Math. \textbf{546}
(2011), 39-81.

\bibitem{Car} Y. Carri\`{e}re, \emph{Flots riemanniens}, in \emph{\
Transversal structure of foliations} (Toulouse, 1982), Ast\'{e}risque 
\textbf{116 }(1984), 31-52.

\bibitem{Col}
B. Colbois and J. Dodziuk, \emph{Riemannian metrics with large $\lambda_1$}, Proc. Amer. Math.
Soc. \textbf{122} (1994), no. 3, 905-906.

\bibitem{D1}
X. Dai,
\emph{Adiabatic limits, non-multiplicity of signature and the Leray spectral sequence},
J. Amer. Math. Soc. \textbf{4} (1991), no. 2, 265-321.

\bibitem{Dom} D. Dom\'{\i}nguez, \emph{Finiteness and tenseness theorems for
Riemannian foliations}, Amer. J. Math. \textbf{120} (1998), no. 6, 1237-1276.

\bibitem{Fuk}
K. Fukaya, \emph{Collapsing of Riemannian manifolds and eigenvalues of Laplace operator},
Invent. Math. \textbf{87} (1987), 517-547.

\bibitem{HabThesis} G. Habib, \emph{Energy-momentum tensor on foliations},
J. Geom. Phys. \textbf{57 }(2007), 2234-2248.

\bibitem{HabRic} G. Habib and K. Richardson, \emph{A brief note on the
spectrum of the basic Dirac operator}, Bull. London Math. Soc. \textbf{41}
(2009), 683-690.

\bibitem{HabRic2} G. Habib and K. Richardson, \emph{Modified differentials
and basic cohomology for Riemannian foliations}, J. Geom. Anal. \textbf{23}
(2013), no. 3, 1314-1342.

\bibitem{EKG} A. El Kacimi-Alaoui and B. Gmira, \emph{Stabilit\'{e} du caract%
\`{e}re k\"{a}hl\'{e}rien transverse}, Israel J. Math. \textbf{101 }(1997),
323-347.

\bibitem{EK} A. El Kacimi-Alaoui, \emph{Op\'{e}rateurs transversalement
elliptiques sur un feuilletage riemannien et applications}, Compositio Math. 
\textbf{73 }(1990), 57-106.

\bibitem{GLP} P. B. Gilkey, J. V. Leahy, J. H. Park, Eigenvalues of the form valued Laplacian
for Riemannian submersions, Proc. Amer. Math. Soc. 
\textbf{126} (1998), no. 6, 1845-1850.


\bibitem{Jam} P. Jammes, \emph{Effondrement, spectre et propri\'{e}t\'{e}s
diophantiennes des flots riemanniens}, Ann. Inst. Fourier (Grenoble) \textbf{%
60} (2010), no. 1, 257-290.

\bibitem{KY}
Y. A. Kordyukov and A. A. Yakovlev,
\emph{Adiabatic limits and the spectrum of the Laplacian on foliated manifolds},
C?-algebras and elliptic theory II, Trends Math., Birkhäuser, Basel, 2008, 123?144.


\bibitem{Lott} J. Lott, \emph{Collapsing and Dirac-type operators}, Geom.
Dedicata \textbf{91} (2002), 175-196.

\bibitem{MaMe} R. R. Mazzeo and R. B. Melrose, 
\emph{The adiabatic limit, Hodge cohomology, and Leray's spectral sequence
for a fibration}, J. Diff. Geom. \textbf{31} (1990), 185-213.

\bibitem{PaRi} E. Park and K. Richardson, \emph{The basic Laplacian of a
Riemannian foliation}, Amer. J. Math. \textbf{118 }(1996), 1249-1275.

\bibitem{Pfa} F. Pf\"affle, 
\emph{Eigenvalues of Dirac operators for hyperbolic degenerations},
Manuscripta Math. \textbf{116} (2005), no. 1, 1-29. 

\bibitem{Rein} B. Reinhart, \emph{Foliated manifolds with bundle-like metrics%
}, Ann. Math. \textbf{69} (1959), 119-132.

\bibitem{Tond} Ph. Tondeur, \emph{Geometry of foliations}, Monographs in
Mathematics \textbf{90}, Birkh\"{a}user Verlag, Basel 1997.

\bibitem{Wit} E. Witten, 
\emph{Global gravitational anomalies},
Comm. Math. Phys. \textbf{100}
(1985), 197-229.


\end{thebibliography}
\end{document}